\documentclass[printer]{biometrika} 
\pdfoutput=1  

\usepackage{times}
\usepackage{bm}
\usepackage{natbib}

\RequirePackage[OT1]{fontenc}
\RequirePackage{amsmath,amsbsy,amssymb}
\RequirePackage{booktabs}
\RequirePackage{enumerate}
\RequirePackage{graphicx}
\RequirePackage{longtable}
\RequirePackage{placeins}
\RequirePackage{setspace}

\usepackage{macros}

\begin{document}

\received{January 2012}

\markboth{P. O. Perry \and P. J. Wolfe}{Null models}

\title{Null models for network data}

\author{Patrick O.\ Perry}
\affil{Statistics Group, NYU Stern School of Business, New~York,
New~York, 10012, U.S.A. \email{pperry@stern.nyu.edu}}
\author{\and Patrick J.\ Wolfe}
\affil{Department of Statistical Science, University College London, London WC1E 6BT, U.K.
\email{patrick@stats.ucl.ac.uk}}

\maketitle

\begin{abstract}
The analysis of datasets taking the form of simple, undirected graphs
continues to gain in importance across a variety of disciplines.  Two choices
of null model, the logistic-linear model and the implicit log-linear model,
have come into common use for analyzing such network data, in part because
each accounts for the heterogeneity of network node degrees typically observed
in practice. Here we show how these both may be viewed as instances of a
broader class of null models, with the property that all members of this class
give rise to essentially the same likelihood-based estimates of link
probabilities in sparse graph regimes.  This facilitates likelihood-based
computation and inference, and enables practitioners to choose the most
appropriate null model from this family based on application context.
Comparative model fits for a variety of network datasets demonstrate the
practical implications of our results.
\end{abstract}


\section{Introduction}

Statisticians have long recognized the importance of so-called null models.  There
are two main uses for null models: (1) they serve as baseline points of
comparison for assessing goodness of fit (e.g., for score tests and analysis
of variance); and (2) they facilitate residuals-based analyses (e.g., for
exploratory data analysis and outlier detection).

In contexts where data take the form of a simple, undirected network on $n$
nodes, the model posited by \cite{erdos1959random}, considered with edges appearing as
independently and identically distributed Bernoulli trials, is perhaps the
simplest possible null model.  However, with only a single parameter, it lacks
the ability to capture the extent of degree heterogeneity commonly associated
with network data in practice \citep{barabasi1999emergence}.  As alternatives,
two popular $n$-parameter models have emerged in the literature, each of which
associates a single parameter (or estimate thereof) to every node, and in
doing so allows for heterogeneity of nodal degrees.

    \emph{The logistic-linear model}
    takes the probability $p_{ij}$ of observing
    an edge between nodes $i$~and~$j$ to be given by
    \[
      \logit p_{ij} = \alpha_i + \alpha_j,
    \]
    where $\valpha = (\alpha_1, \dotsc, \alpha_n)$ is a vector of node-specific parameters.
    \cite{chatterjee2011random} term this the $\beta$-model; it has
    also been considered by \cite{park2004statistical} and \cite{blitzstein2011sequential}, and before them by \cite{holland1981exponential} in its
    directed form.  See also \cite{hunter2004mm} and \cite{rinaldo2011maximum}, and references therein.

    \emph{The (implicit) log-linear model}
    instead takes edge probabilities to be given in terms of an observed binary, symmetric adjacency matrix $\mX$ as
    \[
      \log p_{ij} = \log X_{i+} + \log X_{j+} - \log X_{++},
    \]
    where $X_{i+} = \sum_{k=1}^n X_{ik}$ is the degree of the $i$th node, and $X_{++} = \sum_{i=1}^n X_{i+}$ the sum of all observed degrees.  This model is implicit in that its specified
    edge probabilities depend on the observed data; thus, it
    is not a proper null model.  It is more accurate to say that each
    $p_{ij}$ here is the \emph{estimated} probability under a model, and that
    the model has been left unspecified.
    \cite{girvan2002community} take this as a basis for their residuals-based approach to
    community detection or nodal partitioning in networks, while \cite{chung2003spectra} and others have studied its associated spectral graph properties.

Both the logistic-linear and log-linear models have appealing features.  From
a statistical standpoint, the former is more convenient since it is based on
the canonical link function, whereas from an analytical and computational
standpoint, the latter is more convenient since the set of estimated edge
probabilities takes the form of an outer product.
However, at the same time, the choice between them remains unsatisfying.
Practitioners lack the
necessary guidance to judge which of these two null models is most appropriate
in a given context, along with a clear understanding of the differences
between them.

In the sequel we resolve this issue and show that, in the sparse adjacency
regimes wherein network datasets are typically observed, the two models are
equivalent for all practical purposes.  Specifically, both models lead to
essentially the same parameter estimates $\hat \alpha_i$, the same
probability estimates $\hat p_{ij}$, and the same null log-likelihood
$\hat \ell$.  In fact, by considering these two models
as members of a broader class of null models, we prove the stronger result
that all models in this family lead to essentially the same maximum likelihood
estimates for the parameters of interest.
We emphasize that our results hold irrespective of the data-generating
mechanism giving rise to $\mX$.  Specifically, they hold whenever the
node degrees $X_{i+}$ are small relative to the total number of edges $X_{++}/2$ in the network.

\section{Statement of results}
\subsection{A family of null models for network data}

As above let $\mX$ be an $n \times n$ binary, symmetric adjacency matrix with zeros along its main diagonal, corresponding to a simple, undirected graph on $n$ nodes.  Consider $\mX$ to be random and
$\valpha = (\alpha_1, \dotsc, \alpha_n)$ a vector of node-specific
parameters.
We suppose that $\mX$ has independent Bernoulli elements above the main diagonal
such that $\Prob(X_{ij} = 1) = p_{ij}(\valpha)$, and specify a corresponding family of probabilistic null models for $\mX$, each parameterized by $\valpha$.

To this end, let $\varepsilon = \{ \varepsilon_{ij} : i \neq j \}$ be a
family of smooth functions,
where $\varepsilon_{ij}$ maps pairs of real numbers to real numbers, and
$\varepsilon_{ij}(x, y)~=~\varepsilon_{ji}(y, x)$.
Let model $\mathcal{M}_\varepsilon$ then specify $p_{ij}$ as
\begin{equation}\label{E:null-general}
  \mathcal{M}_\varepsilon
    \!:\,\,
      \log p_{ij}
        = \alpha_i + \alpha_j + \varepsilon_{ij}(\alpha_i, \alpha_j),
\end{equation}
so that we obtain a class of log-linear models indexed by $\varepsilon$.

Observe that with $\varepsilon_{ij}(\alpha_i, \alpha_j)$ identically zero we recover the log-linear null model alluded to in the Introduction, explicitly parameterized by $\valpha$.  In fact, this class encompasses three common choices of link function:
\begin{subequations}
\label{E:null-prob-links}
\begin{align}
  \mathcal{M}_\text{log}
    \!&:\,\, \log p_{ij} = \alpha_i + \alpha_j \label{E:null-log-prob}, \\
  \mathcal{M}_\text{cloglog}
    \!&:\,\, \log(-\log(1 -  p_{ij})) = \alpha_i + \alpha_j \label{E:null-cloglog-prob}, \\
  \mathcal{M}_\text{logit}
    \!&:\,\, \logit p_{ij} = \alpha_i + \alpha_j \label{E:null-logit-prob}.
\end{align}
\end{subequations}
To see this, set
\(
  \varepsilon_{ij}(\alpha_i, \alpha_j)
    = \log\{1 - \exp(-e^{\alpha_i+\alpha_j})\} - (\alpha_i + \alpha_j)
\)
for the complementary log-log link model $\mathcal{M}_\text{cloglog}$, and for the model $\mathcal{M}_\text{logit}$, set
\(
  \varepsilon_{ij}(\alpha_i, \alpha_j) = - \log\{1 + \exp(\alpha_i + \alpha_j)\}.
\)

As we have seen, the logit-link model $\mathcal{M}_\text{logit}$ is an undirected version of
\citefullauthor{holland1981exponential}'s~(\citeyear{holland1981exponential})
exponential family random graph model.  As noted by \cite{chatterjee2011random}, the degree sequence of $\mX$ is sufficient for $\valpha$ in this case, formalizing the null-model intuition that all graphs exhibiting the same degree sequence be considered as equally likely.
The log-link model $\mathcal{M}_\text{log}$ can be considered as an
alternate parametrization of
\citefullauthor{chung2002connected}'s~(\citeyear{chung2002connected})
expected degree model, with the additional constraint that self-loops of the form $X_{ii} = 1$ are explicitly disallowed.  The complementary log-log model $\mathcal{M}_\text{cloglog}$ has not seen application in the network literature to date, but the same functional form appears commonly in the context of generalized linear models \citep{muccullagh1989generalized}.

\subsection{Properties}

Before placing conditions on $\varepsilon$, it is instructive to consider key properties of the family of models that can be written in the form $\mathcal{M}_\varepsilon$.
Observe first that the expected degree of node $i$ is
\begin{equation}
  \E(X_{i+}) = \sum_{j \neq i} p_{ij}  \label{E:chung-lu-degs}.
\end{equation}
Thus, for models in this class, higher values of $\alpha_i$ lead to
higher expected degrees for node~$i$ whenever $\varepsilon$ is small.  It is therefore natural to posit as a simple estimator of $\valpha$ a monotone transformation of the degree sequence specified by $\mX$:
\begin{equation}
  \approxp{\alpha}_i = \log X_{i+} - \log \sqrt{X_{++}} \label{E:chung-lu-alphai}.
\end{equation}
This is equivalent to estimating $p_{ij}$ via
\begin{equation}
  \approxp{p}_{ij} = \frac{X_{i+} X_{j+}}{X_{++}} \label{E:chung-lu-pij},
\end{equation}
which corresponds precisely to the implicit log-linear model described in the Introduction, and may also be viewed in light of~\eqref{E:chung-lu-degs} as an approximate moment-matching technique.

We show below that $\vtalpha$ as defined in~\eqref{E:chung-lu-alphai} suffices as an estimator for this model class in the sparse graph regime.  To gain intuition into this claim, consider the corresponding likelihood function as follows.

The log-likelihood for any simple, undirected graph model with independent edges is
\begin{equation}
  \log\Bigl\{
    \prod_{i < j}
      X_{ij}^{p_{ij}}
      (1 - X_{ij})^{1 - p_{ij}}
  \Bigr\}
  =
  \sum_{i < j}
      X_{ij} \log p_{ij}
      +
      (1 - X_{ij}) \log (1 - p_{ij}) \label{E:Bern-log-lik}.
\end{equation}
Intuitively, if $p_{ij}$ is small, then a Bernoulli random variable
with mean $p_{ij}$ behaves like a Poisson random variable having the
same mean.  In a rough sense, this Bernoulli log-likelihood is close to one under
which $X_{ij}$ is treated as a Poisson random variable:
\begin{equation*}
  \log\Bigl\{
    \prod_{i < j}
      \frac{p_{ij}^{X_{ij}}}{X_{ij}!}
      \exp(-p_{ij})
  \Bigr\}
  =
  \sum_{i < j}
    X_{ij} \log p_{ij} - p_{ij},
\end{equation*}
up to a constant shift depending on $\mX$.

Substituting the log-linear model parametrization $\mathcal{M}_\text{log}$ of~\eqref{E:null-log-prob}, which corresponds to the canonical link under Poisson sampling, we obtain
\begin{align*}
  \ell_\text{Pois}(\valpha)
    &=
      \sum_{i < j}
        X_{ij} (\alpha_i + \alpha_j)
        - \exp(\alpha_i + \alpha_j) \\
    &=
      \sum_{i=1}^n
        \alpha_i \, X_{i+}
      -
      \sum_{i \neq j}
        \exp(\alpha_i + \alpha_j),
\end{align*}
and thus, the solution to the Poisson likelihood equation
$\nabla \ell_\text{Pois}(\valpha) = 0$ satisfies
\[
  X_{i+} =
    \sum_{j \neq i}
      \exp(\alpha_i + \alpha_j) \quad (i = 1,\ldots,n).
\]
When $\valpha$ is set to $\vtalpha$ as defined in~\eqref{E:chung-lu-alphai}, the right-hand side becomes
\begin{align*}
  \sum_{j \neq i} \exp(\approxp{\alpha}_i + \approxp{\alpha}_j)
    &= \frac{X_{i+}}{X_{++}} \sum_{j \neq i} X_{j+} \\
    &= X_{i+}\left(1 - \frac{X_{i+}}{X_{++}}\right),
\end{align*}
and so we see that each component of $\nabla \ell_\text{Pois} (\vtalpha)$ is precisely $X_{i+}^2/X_{++}$.  Hence, when this quantity is small for every $i$, we can expect that both $\approxp{\valpha}$ and correspondingly $\approxp{p}$ are close to their respective maximum
likelihood estimates.  We formalize this notion as follows.

\subsection{Approximation results for maximum likelihood inference}

Our main result is an approximation theorem for likelihood-based inference under models taking the form of $\mathcal{M}_\varepsilon$ from~\eqref{E:null-general}.  Under suitable sparsity constraints and for many choices of $\varepsilon$, including those given by~\eqref{E:null-prob-links}, a maximum likelihood estimate of each parameter $\alpha_i$ exists and is close to $\approxp{\alpha_i}$ as defined in~\eqref{E:chung-lu-alphai}.  Furthermore, the corresponding maximum likelihood estimate of each edge probability $p_{ij}$ is close to $\approxp{p}_{ij}$, defined in~\eqref{E:chung-lu-pij}, and the null log-likelihood under $\mathcal{M}_\varepsilon$ evaluated at $\approxp{p}$ is close to that evaluated at the corresponding maximum likelihood estimate of $p$.

These approximation results hold for all $\varepsilon$ satisfying the following condition.
\begin{assumption}\label{A:small-eps}
  For all pairs $i,j$ and all choices of $k$, $l$, and $m$, the functions
  \(
    \varepsilon_{ij},
  \)
  \(
    \partial \varepsilon_{ij} / \partial \alpha_k,
  \)
  \(
    \partial^2 \varepsilon_{ij} / (\partial \alpha_k \partial \alpha_l),
  \)
  and
  \(
    \partial^3 \varepsilon_{ij} / (\partial \alpha_k \partial \alpha_l \partial \alpha_m),
  \)
  are sub-exponential in $\alpha_i + \alpha_j$.  That is, there
  exists a constant $C_0$ such that the absolute values of these functions are
  bounded by $C_0 \exp(\alpha_i + \alpha_j)$.
\end{assumption}

\noindent
Recall model $\mathcal{M}_\text{log}$ from~\eqref{E:null-log-prob}, for which $\varepsilon$ is identically zero and thus satisfies Assumption~\ref{A:small-eps} with $C_0 = 0$. One can show that the specifications of $\varepsilon$ arising in the $\mathcal{M}_\text{cloglog}$ and $\mathcal{M}_\text{logit}$ models, from~\eqref{E:null-cloglog-prob} and~\eqref{E:null-logit-prob} respectively, satisfy Assumption~\ref{A:small-eps} with $C_0$ equal to $1/2$ and $1$.

Our sparsity requirement is that each component $X_{i+}^2/X_{++}$ of $\nabla \ell_\text{Pois} (\vtalpha)$ be sufficiently small; for example, $15^{-2}$ in the case of the log-link model $\mathcal{M}_\text{log}$.  We then have the following.

\begin{theorem}\label{T:null-mle-approx}
  Suppose $\mX$ is an $n \times n$ adjacency matrix such that
  \(
    1 \leq X_{i+}^2 \leq \varepsilon_0 \, X_{++}
  \)
  for all $i$.
  For some set of smooth functions $\varepsilon$ satisfying
  Assumption~\ref{A:small-eps}, let model $\mathcal{M}_\varepsilon$
  with parameter vector $\valpha$ in $\reals^n$
  specify edge probability $p_{ij} = p_{ij}(\valpha)$
  as in \eqref{E:null-general}.  Let $\vtalpha$ be as defined
  in~\eqref{E:chung-lu-alphai}.

  Define
  \(
    \bar \varepsilon_0 = \{15 \, (C_0 + 1)\}^{-2}
  \)
  and
  \(
    C = 10 \, (C_0 + 1),
  \)
  where $C_0$ is as in Assumption~\ref{A:small-eps}.
  If $\varepsilon_0 \leq \bar \varepsilon_0$ then there exists
  a solution to the likelihood equation, $\vhalpha$, such that
  \[
    \| \vhalpha - \vtalpha \|_\infty
      \leq  C \, \varepsilon_0.
  \]
\end{theorem}

As shown in the Appendix, the corresponding approximation result for the maximum likelihood estimate of $p_{ij}$ is a straightforward consequence, and an approximation for the log-likelihood itself also follows.

\begin{corollary}\label{C:null-pij-approx}
  Suppose the conditions of Theorem~\ref{T:null-mle-approx} hold and that
  $\hat p_{ij} = p_{ij}(\vhalpha)$.
  If $\varepsilon_0 \leq \bar \varepsilon_0$, then
  \[
    \left|\frac{\hat p_{ij} - \approxp{p}_{ij}}{\approxp{p}_{ij}}\right|
      \leq C_1 \, \varepsilon_0,
  \]
  where $C_1 = 24 \, (C_0 + 1)$.
\end{corollary}

\begin{corollary}\label{C:null-ll-approx}
  Suppose the conditions of Theorem~\ref{T:null-mle-approx} hold, that
  $\hat \ell$ is the log-likelihood under $\mathcal{M}_\varepsilon$, evaluated
  at $\hat p$, and that $\tilde \ell$ is defined analogously as
  \[
    \tilde \ell = \sum_{i < j} X_{ij} \log \tilde p_{ij} - (1 - X_{ij}) \log (1 - \tilde p_{ij}).
  \]
  If $\varepsilon_0 \leq \bar \varepsilon_0$, then
  \[
    \left|\frac{\hat \ell - \tilde \ell}{\tilde \ell}\right|
      \leq C_2 \, \varepsilon_0,
  \]
  where $C_2 = 49 \, (C_0 + 1).$
\end{corollary}

Notably, the results of Theorem~\ref{T:null-mle-approx} and
Corollaries~\ref{C:null-pij-approx} and~\ref{C:null-ll-approx} are not probabilistic.
The only assumption on $\mX$ is that the nodal degrees are nonzero and
small relative to the total number of edges.  Thus, these results hold even when
the true model for $\mX$ is not specified by
$\mathcal{M}_\varepsilon$; i.e., they are robust to model
misspecification.

\section{Proof of Theorem~\ref{T:null-mle-approx}}\label{S:proof-of-theorem}

We now outline the proof of Theorem~\ref{T:null-mle-approx}, deferring requisite technical lemmas to the Appendix.  We employ \citefullauthor{kantorovich1948functional}'s (\citeyear{kantorovich1948functional}) analysis of Newton's method, specifically the optimal error bounds given by \cite{gragg1974optimal}, whose notation we adopt below for ease of reference.

Our strategy is to use the Kantorovich Theorem to establish the existence of a maximum likelihood estimate $\vhalpha$ of $\valpha$ in a neighborhood of $\vtalpha$, which in turn can be obtained by applying Newton's method with $\vtalpha$ as the initial point.  If we are able to establish the necessary hypotheses, then this theorem will enable us to bound the distance between $\vtalpha$ and $\vhalpha$ as required.  To apply it we require a Lipschitz condition on the Jacobian of the corresponding system of equations near $\vtalpha$, as well as boundedness conditions on the inverse Hessian evaluated at $\vtalpha$ and also the initial step size of Newton's method from $\vtalpha$.

As we show below, the key to these conditions is an approximation of the Hessian by a diagonal-plus-rank-1-matrix formed from $\vtalpha$.  First, recall the data log-likelihood under $\mathcal{M}_\varepsilon$ from~\eqref{E:Bern-log-lik}; its gradient and Hessian with respect to $\valpha$ may be written component-wise as
  \begin{align*}
    \frac{\partial \ell}{\partial \alpha_k}
      & =
      \sum_{j \neq k}
        (X_{kj} - e^{\alpha_k + \alpha_j})
      +
      \sum_{j \neq k}
        X_{kj} f_{kj}
      +
      \sum_{j \neq k}
        e^{\alpha_k + \alpha_j} \bar f_{kj},
      \\
      \frac{\partial^2 \ell}{\partial \alpha_k \partial \alpha_l}
      & =
      \begin{cases}
      - e^{\alpha_k + \alpha_l}
      + X_{kl} \frac{\partial f_{kl}}{\partial \alpha_l}
      + e^{\alpha_k + \alpha_l}
        \Big(\bar f_{kl} + \frac{\partial \bar f_{kl}}{\partial \alpha_l}\Big)
      & \text{if $k \neq l$,}
      \\
      - \sum_{j \neq k} e^{\alpha_k + \alpha_j}
      + \sum_{j \neq k}
          X_{kj} \frac{\partial f_{kj}}{\partial \alpha_k}
      + \sum_{j \neq k}
          e^{\alpha_k + \alpha_j}
          \Big(\bar f_{kj} + \frac{\partial \bar f_{kj}}{\partial \alpha_k}\Big)
      & \text{if $k = l$,}
      \end{cases}
  \end{align*}
  where
  \begin{gather*}
    f_{ij}
      = f_{ij}(\alpha_i, \alpha_j)
      = \frac{\partial \varepsilon_{ij}}{\partial \alpha_i}
        + \frac{p_{ij}}{1 - p_{ij}}
          \Big(1 + \frac{\partial \varepsilon_{ij}}{\partial \alpha_i}\Big),
    \\
    \bar f_{ij}
      = \bar f_{ij}(\alpha_i, \alpha_j)
      = 1 - \exp(\varepsilon_{ij}) - \exp(\varepsilon_{ij}) f_{ij}.
  \end{gather*}

  The form of these expressions suggests that when $\varepsilon_{ij}$ and $p_{ij}$ are small, and $f_{ij}, \bar f_{ij}$ and their derivatives controlled, an approximation of $\nabla^2 \ell(\valpha)$ based on terms $\approxp{p}_{ij} = \exp(\approxp{\alpha}_i + \approxp{\alpha}_j)$  will be effective in a neighborhood of $\vtalpha$.   
  Defining such a neighborhood parameterized by $r \geq 1$ as $\mathcal{N}_r = \{ \valpha : \| \valpha - \vtalpha \|_\infty \leq (\log r) / 2 \}$, Lemmas 1--5 in the Appendix provide the necessary approximation bounds as a function of $r$.

  Now define vector $\vd = (d_1, \ldots, d_n)$ with $d_i = X_{i+}$, and, setting
  $\mD = \diag(\vd)$, write
  \[
    \nabla^2 \ell(\vtalpha)
     =  \mH
        \big(\mI + \mE \big),
  \]
  where we choose $\mH$ to be
  \[
    \mH = -\Big(\mD + \frac{1}{X_{++}} \vd \vd^\trans \Big).
  \]
  The Sherman-Morrison formula gives
  \[
    \mH^{-1}
    =
       -\mD^{-1} + \frac{1}{2X_{++}} \vone \vone^\trans,
  \]
  and with this expression we may bound the norm of $\mE$, according to Lemma~\ref{L:E-norm-bounds} in the Appendix.

  To conclude the proof of Theorem~\ref{T:null-mle-approx},
  consider the system of equations
  \(
    \mF(\vx) = \mD^{-1} [ \nabla \ell(\vx) ]
  \)
  with derivative matrix
  \(
    \mF'(\vx) = \mD^{-1} [ \nabla^2 \ell(\vx) ].
  \)
  Equipped with the norm $\| \cdot \|_\infty$, Lemmas~\ref{L:inv-hess} and~\ref{L:newton-step} then establish the bounding constants $\varkappa$ and $\delta$, and Lemma~\ref{L:hess-lipschitz} the Lipschitz constant $\lambda$, necessary to apply Kantorovich's result, taking initial Newton iterate $\vx_0 = \vtalpha$ and defining
  subsequent iterates recursively by
  \(
    \vx_{k+1}
      = \vx_k - [\mF'(\vx_k)]^{-1} \mF(\vx_k)
      = \vx_k - [\nabla^2 \ell(\vx_k)]^{-1} \nabla \ell(\vx_k).
  \)
  Lemmas~\ref{L:inv-hess} and~\ref{L:newton-step} require $\| \mE \|_\infty < 1$, and 
  Lemma~\ref{L:newton-step} further relates $\delta = L_1 \, \varkappa \, \varepsilon_0$, with constant $L_1$ defined in Lemma~\ref{L:grad-hess-null-bounds}.

  If we define
  \(
     h = 2 \varkappa \lambda \delta
  \)
  and
  \(
    t^\ast = (2/h)(1 - \sqrt{1 - h}) \, \delta,
  \)
  then Kantorovich's Theorem asserts that when $h \leq 1$ and $t^\ast \leq (\log r)/2$,
  each iterate $\vx_k$ is
  in $\mathcal{N}_r$ and $\vx^\ast$ is well defined, in the sense that as $k$ increases, $\vx_k$
  converges to a limit point $\vx^\ast$ such that $\mF(\vx^\ast) = 0$.
  The matrix $\mD$ is of full rank, and so in this case $\nabla \ell(\vx^\ast) = 0$ as
  well, implying that $\vx^\ast$ is a solution to the likelihood equation.  Thus we will take $\vhalpha = \vx^\ast$, and the existence of a maximum likelihood estimate will be established if we can show that $h \leq 1$ and $t^\ast \leq (\log r)/2$.
    
  To show this, set $r = \exp(4 \, \delta)$, so that $t^\ast \leq 2 \delta = (\log r)/2$.
  It is then straightforward to verify that if $\varepsilon_0 \leq \bar \varepsilon_0 = \{15 \, (C_0 + 1)\}^{-2}$, where $C_0$ is as given in Assumption~\ref{A:small-eps}, then $\| \mE \|_\infty \leq 10 \, (C_0 + 1) \, \varepsilon_0 < 1$, satisfying the requirements of Lemmas~\ref{L:inv-hess} and~\ref{L:newton-step}.  Moreover, if also $r = \exp(4 \, \delta)$, then $\lambda \leq 16 \, (C_0 + 1)$, $h \leq 1$, and $L_1 \, \varkappa \leq 5 \, (C_0 + 1)$.
  By \cite{gragg1974optimal},
  \(
    \| \vx^\ast - \vx_k \|_\infty
      \leq
        2^{-k + 1} \, \| \vx_1 - \vx_0 \|_\infty,
  \)
  and so the result of Theorem~\ref{T:null-mle-approx} then follows, since
  \(
    \| \vx_1 - \vx_0 \|_\infty \leq \delta = L_1 \, \varkappa \, \varepsilon_0.
  \)

\section{Discussion}\label{S:discussion}

\subsection{Implications}

The main implication of the above results is that a broad class of null models for undirected networks give rise to roughly the same maximum likelihood estimates of edge probabilities. For practitioners looking to capture degree heterogeneity in their null models, then, this provides verifiable assurances that the particular choice of null model will not give meaningfully different conclusions, provided the null model can be written in the form $\mathcal{M}_\varepsilon$ such that Assumption~\ref{A:small-eps} is satisfied, and the dataset is sufficiently sparse. Empirically, as we show below, our approximation bounds and sparsity conditions appear conservative in practice.  

When comparing to the extant literature, two recent results warrant discussion.  First is that of \cite{chatterjee2011random}, in which the authors show that a unique maximum likelihood estimate exists with high probability when the model $\mathcal{M}_\text{logit}$ is in force, and give an iterative algorithm that converges geometrically quickly when a solution to the likelihood equation exists.  In contrast, our results are deterministic, and do not require the data to be generated by any particular model.

\cite{rinaldo2011maximum} also address the model $\mathcal{M}_\text{logit}$ and its version for directed graphs, focusing on necessary and sufficient conditions for the existence of a maximum likelihood estimate, and the failure thereof, as a function of the polytope of admissible degree sequences for a given network size.  As with \cite{chatterjee2011random}, however, their existence results are probabilistic in nature; one interpretation of our results in this context is that our sparsity conditions are sufficient to avoid the pathological degree polytope conditions that give rise to the many nonexistence examples considered by \cite{rinaldo2011maximum}.

\subsection{Empirical evaluation}\label{S:empirical}

To evaluate how conservative our sparsity condition and bound on the universal constant $C$ in Theorem~\ref{T:null-mle-approx} appear in practice, we fitted the models $\mathcal{M}_\text{log}$, $\mathcal{M}_\text{cloglog}$, and $\mathcal{M}_\text{logit}$ defined in (\ref{E:null-log-prob}--\ref{E:null-logit-prob}) to nine different network datasets of sizes ranging from $n = 34$ to $n = 7610$:

{\footnotesize
\begin{longtable}{l p{30em}}
  \cite{zachary1977information} & Social ties within a college karate club \\
  \cite{girvan2002community} & Network of American football games between Division IA colleges \\
  \cite{hummon1990analyzing} & Citations among scholarly papers on the subject of centrality in networks \\
  \cite{gleiser2003community} & Collaborations between jazz musicians \\
  \cite{duch2005community} & Metabolic network of \textit{C. elegans} \\
  \cite{adamic2005political} & Hyperlinks between weblogs on US politics \\
  \cite{newman2006finding} & Coauthorships amongst researchers on network
    theory and experiments \\
  \cite{watts1998collective} & Topology of the Western States Power Grid of the U.S.A. \\
  \cite{newman2001structure} & Coauthorships among postings to a High-Energy Theory preprint archive
\end{longtable}
}\vspace{-\baselineskip}%

\noindent
In each case we obtained a maximum-likelihood estimate $\vhalpha$ and compared it to $\vtalpha$. Figure~\ref{F:hep-th-approx} shows the approximation errors in $\hat
\valpha_i$ for the dataset analyzed by \cite{newman2001structure} as a typical example, and Table~\ref{T:cl-approx} summarizes results for all nine datasets.
\begin{figure}
    \centering
    \makebox{\includegraphics[scale=0.6]{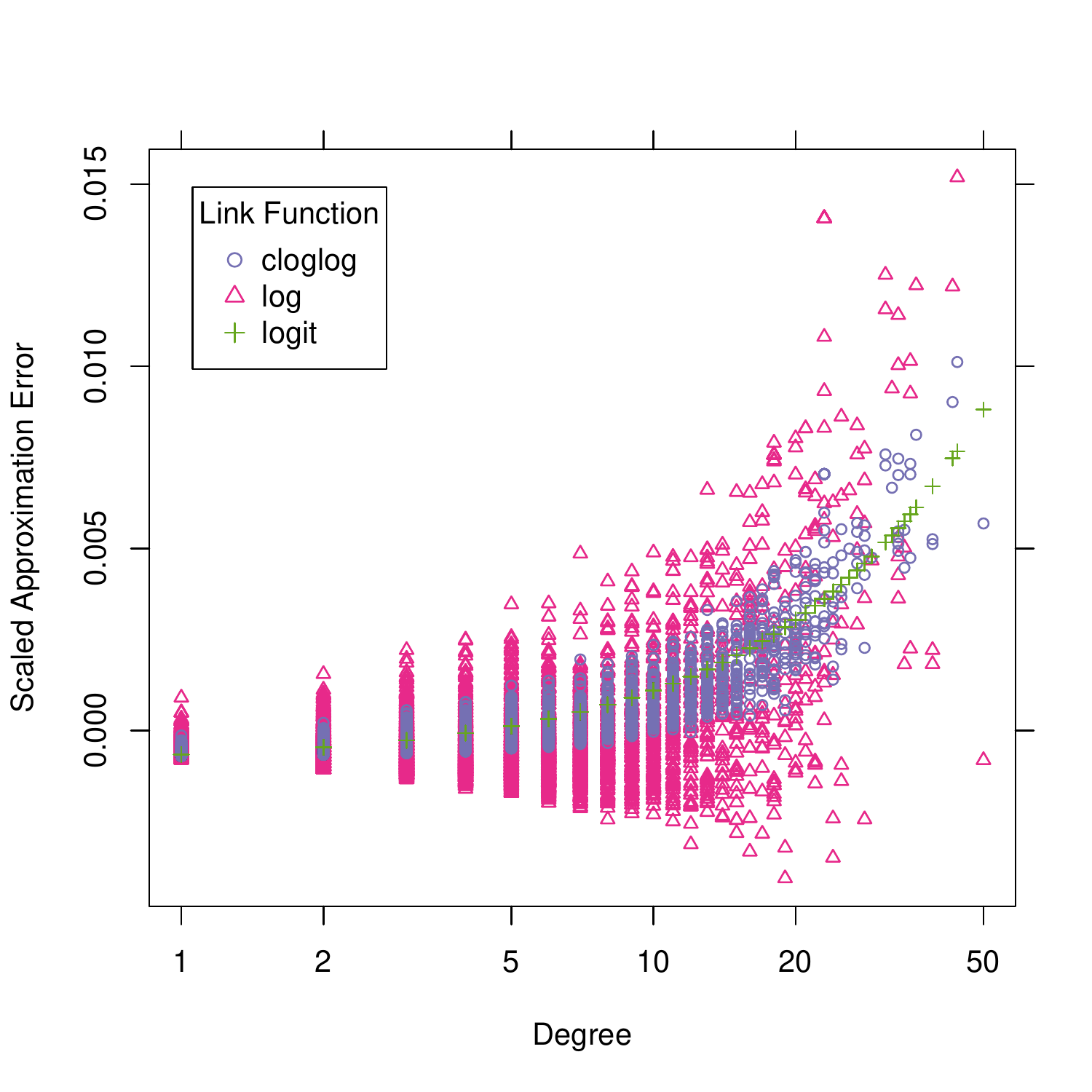}}
    \caption{ Scaled approximation error
      $(\hat \alpha_i - \approxp{\alpha}_i) / (C \varepsilon_0)$ plotted as a
      function of degree $X_{i+}$ for the network analyzed by \cite{newman2001structure},
      using the $\cloglog$, $\log$, and $\logit$ links. }\label{F:hep-th-approx}
\end{figure}
\begin{table}
  \caption{
    Approximation error in terms of $\varepsilon_0$ for several datasets.  Percentage valid
    is defined as $(100/n) \sum_{i} I(X_{i+}^2 / X_{++} \leq \bar
    \varepsilon_0)$.  For Theorem~\ref{T:null-mle-approx} to apply,
    this value should be equal to $100$; however, the corresponding approximation results hold across the range of datasets considered.
  }
  \label{T:cl-approx}
  \tiny
  \centering
  \singlespace
  \makebox[0.66\textwidth]{\begin{tabular}{lrrrlrrr}
\toprule
\textbf{Dataset}
        & \multicolumn{1}{c}{$n$}
        & \multicolumn{1}{c}{$X_{++}$}
        & \multicolumn{1}{c}{$\max X_{i+}$}
        & \textbf{Link}
        & \textbf{Valid \%}
        & \multicolumn{1}{c}{$\frac{\|\vhalpha - \vtalpha\|_{2}}{\sqrt{n} \, C \, \varepsilon_0}$}
        & \multicolumn{1}{c}{$\frac{\|\vhalpha - \vtalpha\|_{\infty}}{C \, \varepsilon_0}$}
 \\
\midrule
\cite{zachary1977information}               &     34 &    156 &     17 \phantom{$i+$} & $\cloglog$ & 0 \hphantom{ \%} &  0.004 \phantom{$\|_2$} &   0.01 \phantom{$\|_\infty$} \\
                               &        &        &        \phantom{$i+$} & $\log$ & 0 \hphantom{ \%} &  0.006 \phantom{$\|_2$} &   0.02 \phantom{$\|_\infty$} \\
                               &        &        &        \phantom{$i+$} & $\logit$ & 0 \hphantom{ \%} &  0.009 \phantom{$\|_2$} &   0.03 \phantom{$\|_\infty$} \\
\addlinespace
\cite{girvan2002community}             &    115 &   1226 &     12 \phantom{$i+$} & $\cloglog$ & 0 \hphantom{ \%} &   0.02 \phantom{$\|_2$} &   0.02 \phantom{$\|_\infty$} \\
                               &        &        &        \phantom{$i+$} & $\log$ & 0 \hphantom{ \%} &  0.005 \phantom{$\|_2$} &   0.01 \phantom{$\|_\infty$} \\
                               &        &        &        \phantom{$i+$} & $\logit$ & 0 \hphantom{ \%} &   0.02 \phantom{$\|_2$} &   0.03 \phantom{$\|_\infty$} \\
\addlinespace
\cite{hummon1990analyzing}           &    118 &   1226 &     66 \phantom{$i+$} & $\cloglog$ & 10 \hphantom{ \%} &  0.003 \phantom{$\|_2$} &   0.01 \phantom{$\|_\infty$} \\
                               &        &        &        \phantom{$i+$} & $\log$ & 19 \hphantom{ \%} &  0.002 \phantom{$\|_2$} &   0.01 \phantom{$\|_\infty$} \\
                               &        &        &        \phantom{$i+$} & $\logit$ & 10 \hphantom{ \%} &  0.004 \phantom{$\|_2$} &   0.02 \phantom{$\|_\infty$} \\
\addlinespace
\cite{gleiser2003community}                 &    198 &   5484 &    100 \phantom{$i+$} & $\cloglog$ & 6 \hphantom{ \%} &  0.004 \phantom{$\|_2$} &   0.02 \phantom{$\|_\infty$} \\
                               &        &        &        \phantom{$i+$} & $\log$ & 7 \hphantom{ \%} &  0.002 \phantom{$\|_2$} &   0.02 \phantom{$\|_\infty$} \\
                               &        &        &        \phantom{$i+$} & $\logit$ & 4 \hphantom{ \%} &  0.005 \phantom{$\|_2$} &   0.02 \phantom{$\|_\infty$} \\
\addlinespace
\cite{duch2005community}             &    453 &   4050 &    237 \phantom{$i+$} & $\cloglog$ & 5 \hphantom{ \%} &  5e-04 \phantom{$\|_2$} &  0.004 \phantom{$\|_\infty$} \\
                               &        &        &        \phantom{$i+$} & $\log$ & 36 \hphantom{ \%} &  6e-04 \phantom{$\|_2$} &  0.009 \phantom{$\|_\infty$} \\
                               &        &        &        \phantom{$i+$} & $\logit$ & 5 \hphantom{ \%} &  6e-04 \phantom{$\|_2$} &  0.005 \phantom{$\|_\infty$} \\
\addlinespace
\cite{adamic2005political}             &   1224 &  33430 &    351 \phantom{$i+$} & $\cloglog$ & 42 \hphantom{ \%} &  9e-04 \phantom{$\|_2$} &  0.006 \phantom{$\|_\infty$} \\
                               &        &        &        \phantom{$i+$} & $\log$ & 50 \hphantom{ \%} &  0.001 \phantom{$\|_2$} &   0.02 \phantom{$\|_\infty$} \\
                               &        &        &        \phantom{$i+$} & $\logit$ & 38 \hphantom{ \%} &  0.002 \phantom{$\|_2$} &   0.01 \phantom{$\|_\infty$} \\
\addlinespace
\cite{newman2006finding}           &   1461 &   5484 &     34 \phantom{$i+$} & $\cloglog$ & 63 \hphantom{ \%} &  0.002 \phantom{$\|_2$} &   0.01 \phantom{$\|_\infty$} \\
                               &        &        &        \phantom{$i+$} & $\log$ & 75 \hphantom{ \%} &  0.003 \phantom{$\|_2$} &   0.02 \phantom{$\|_\infty$} \\
                               &        &        &        \phantom{$i+$} & $\logit$ & 46 \hphantom{ \%} &  0.001 \phantom{$\|_2$} &   0.01 \phantom{$\|_\infty$} \\
\addlinespace
\cite{watts1998collective}                &   4941 &  13188 &     19 \phantom{$i+$} & $\cloglog$ & 93 \hphantom{ \%} &  0.001 \phantom{$\|_2$} &   0.01 \phantom{$\|_\infty$} \\
                               &        &        &        \phantom{$i+$} & $\log$ & 97 \hphantom{ \%} &  0.002 \phantom{$\|_2$} &   0.02 \phantom{$\|_\infty$} \\
                               &        &        &        \phantom{$i+$} & $\logit$ & 80 \hphantom{ \%} &  0.001 \phantom{$\|_2$} &   0.01 \phantom{$\|_\infty$} \\
\addlinespace
\cite{newman2001structure}               &   7610 &  31502 &     50 \phantom{$i+$} & $\cloglog$ & 87 \hphantom{ \%} &  9e-04 \phantom{$\|_2$} &   0.01 \phantom{$\|_\infty$} \\
                               &        &        &        \phantom{$i+$} & $\log$ & 94 \hphantom{ \%} &  0.001 \phantom{$\|_2$} &   0.02 \phantom{$\|_\infty$} \\
                               &        &        &        \phantom{$i+$} & $\logit$ & 78 \hphantom{ \%} &  8e-04 \phantom{$\|_2$} &  0.009 \phantom{$\|_\infty$} \\
\bottomrule
\end{tabular}
}
\end{table}

Two empirical confirmations of Theorem~\ref{T:null-mle-approx} are that for these
datasets and models, the supremum norm distances $\| \vhalpha - \vtalpha \|_\infty$ are
of order $C \, \varepsilon_0$, while the Euclidean norm distances $\| \vhalpha - \vtalpha \|_2$ are of order $\sqrt{n} \, C \, \varepsilon_0$, with $\varepsilon_0$ taken to be $\max X_{i+}^2 / X_{++}$ in each case.  Here the corresponding constants appear conservative by
factors of order $10^3$ and~$10^2$, respectively, suggesting that our results may in fact hold under less stringent sparsity conditions.

\subsection{Avenues for future work}

The simple, well known, and computationally convenient estimator featured in our results has conceptual as well as computational advantages. As a monotone transformation of the observed degree sequence, it can also be seen to yield a parametric interpretation of the classical degree centrality ranking metric common in social network analysis.  While we do not pursue this approach further here, we also note that the approximation to the Fisher information arising in our proof of Theorem~\ref{T:null-mle-approx} can be used to obtain an approximate asymptotic covariance expression in this setting, avoiding a requisite matrix inversion that may be prohibitively costly for large datasets.

More generally, we expect that the methods used in the proof of Theorem~\ref{T:null-mle-approx} will also find application in investigations of
\emph{alternative} network models.  For example, one can show that a variant of
Assumption~\ref{A:small-eps} holds for the degree-corrected blockmodel of
\cite{karrer2011stochastic} and variations thereof.  We therefore surmise that it may well be possible to establish similar universality results for these models.

\section*{Acknowledgement}

Work supported in part by the National Science Foundation, National Institutes of Health, Army Research Office and the Office of Naval Research, U.S.A.

\appendix
\section*{Appendix}
\label{S:additional-proofs}

\subsection*{Technical Lemmas for Proof of Theorem~\ref{T:null-mle-approx}}

Recall our earlier definition of $\mathcal{N}_r = \{ \valpha : \| \valpha - \vtalpha \|_\infty \leq (\log r) / 2 \}$, defining a neighborhood of $\vtalpha$ parameterized by $r \geq 1$.  For $\valpha \in \mathcal{N}_r$, Lemmas~\ref{L:eps-bounds}--\ref{L:fij-bounds} provide bounds on $\varepsilon_{ij}$ and $p_{ij}$, as well as their first three partial derivatives, along with $f_{ij}, \bar f_{ij}$ and their partials.  Lemmas~\ref{L:grad-hess-null-bounds} and~\ref{L:hessian-diff} provide approximations for derivatives of the log-likelihood at $\vtalpha$, and bounds on the change in its second derivative in a neighborhood of $\vtalpha$.  The bounds in Lemmas~\ref{L:eps-bounds}--\ref{L:hessian-diff} are straightforward to verify, and thus their proofs are omitted.

  \begin{lemma}\label{L:eps-bounds}
    If $\valpha \in \mathcal{N}_r$, then $\varepsilon_{ij}$ and its first
    three partial derivatives are bounded by $C_0 \, \tilde p_{ij} \, r$.
  \end{lemma}

  \begin{lemma}
    If $\valpha \in \mathcal{N}_r$, then
    $p_{ij} \leq P_0 \, \tilde p_{ij}$, where
    $P_0 = P_0(r) = r \, \exp( C_0 \, \varepsilon_0 \, r )$.
    Furthermore,
    \begin{gather*}
      \Big|\frac{\partial p_{ij}}{\partial \alpha_k}\Big|
        \leq P_1 \, \tilde p_{ij}, \quad
      \Big|\frac{\partial^2 p_{ij}}{\partial \alpha_k \partial \alpha_k}\Big|
        \leq P_2 \, \tilde p_{ij}, \quad
      \Big|\frac{\partial^3 p_{ij}}{\partial \alpha_k \partial \alpha_k \partial \alpha_l}\Big|
        \leq P_3 \, \tilde p_{ij}.
    \end{gather*}
    where
    \(
      P_1 = P_0 \cdot (1 + C_0 \, \varepsilon_0 \, r),
    \)
    \(
      P_2 = P_0 \cdot \{ (1 + C_0 \, \varepsilon_0 \, r)^2 + C_0 \, \varepsilon_0 \, r \},
    \)
    and
    \(
      P_3 = P_0 \cdot \{
        (1 + C_0 \, \varepsilon_0 \, r)^3
	+
        (1 + C_0 \, \varepsilon_0 \, r)
	\,
	C_0 \, \varepsilon_0 \, r
	+
	C_0 \, \varepsilon_0 \, r
      \}.
    \)
  \end{lemma}

  \begin{lemma}\label{L:fij-bounds}
    If $\valpha \in \mathcal{N}_r$, then
    \begin{enumerate}
      \item
        \(
           |f_{ij}| \leq F_0 \, \tilde p_{ij}
        \)
        and
        \(
           |\bar f_{ij}| \leq \bar F_0 \, \tilde p_{ij},
        \)
	where
	\(
	  F_0 = C_0 \, r + \frac{P_1}{1 - P_0 \, \varepsilon_0}
	\)
	and
	\(
	  \bar F_0 = C_0 \, r \, \exp(C_0 \, \varepsilon_0 \, r) + F_0;
	\)
      \item
        \(
          \Big|
            \frac{\partial f_{ij}}{\partial \alpha_k}
          \Big|
            \leq F_1 \, \tilde p_{ij}
        \)
        and
        \(
          \Big|
            \frac{\partial \bar f_{ij}}{\partial \alpha_k}
          \Big|
            \leq \bar F_1 \, \tilde p_{ij},
        \)
       where
        \(
          F_1 = C_0 \, r
	      + \varepsilon_0 \, \left(\frac{P_1}{1 - P_0 \, \varepsilon_0}\right)^2
	      + \frac{P_2}{1 - P_0 \, \varepsilon_0}
        \)
        and
        \(
          \bar F_1 = \exp(C_0 \, \varepsilon_0 \, r) \{ C_0 \, r + C_0 \, F_0 \, \varepsilon_0 \, r + F_1\};
        \)
      \item
        \(
          \Big|
            \frac{\partial^2 f_{ij}}{\partial \alpha_k \partial \alpha_l}
          \Big|
            \leq F_2 \, \tilde p_{ij},
        \)
        where
        \(
          F_2 = C_0 \, r
	      + 2 \, \varepsilon_0^2 \, \left(\frac{P_1}{1 - P_0 \, \varepsilon_0}\right)^3
	      + 3 \, \varepsilon_0 \, \frac{P_2 \, P_1}{[1 - P_0 \, \varepsilon_0]^2}
	      + \frac{P_3}{1 - P_0 \, \varepsilon_0}.
        \)
   \end{enumerate}
  \end{lemma}
  
    \begin{lemma}\label{L:grad-hess-null-bounds}
    The following approximations for the derivatives of the log-likelihood at
    $\vtalpha$ hold:
    \begin{enumerate}
      \item
        \(
          \left|
            \frac{\partial \ell}{\partial \alpha_k}(\vtalpha)
          \right|
            \leq  L_1 \, X_{k+} \, \varepsilon_0,
        \)
        where
        \(
          L_1 = 1 + F_0(1) + \bar F_0(1);
        \)
      \item
        \(
          \left|
            \frac{\partial^2 \ell}
                 {\partial \alpha_k \partial \alpha_l}
              (\vtalpha)
            +
            \approxp{p}_{kl}
          \right|
            \leq L_2 \, X_{kl} \, \varepsilon_0
               + \bar L_2 \, \tilde p_{kl} \, \varepsilon_0,
        \)
        where
        \(
          L_2 = F_1(1)
        \)
        and
        \(
          \bar L_2 = \bar F_0(1) + \bar F_1(1)
        \),
	provided $k \neq l$;
     \item
        \(
          \left|
            \frac{\partial^2 \ell}{\partial \alpha_k^2}(\vtalpha)
            +
            X_{k+}
            +
            \approxp{p}_{kk}
          \right|
            \leq  L_3 \, X_{k+} \, \varepsilon_0,
        \)
        where
        \(
          L_3 = 2 + L_2 + \bar L_2.
        \)
  \end{enumerate}
  \end{lemma}

  \begin{lemma}\label{L:hessian-diff}
    If $\valpha, \valpha' \in \mathcal{N}_r$, then
    \begin{enumerate}
      \item
        \(
          \left|
            \frac{\partial^2 \ell}
                 {\partial \alpha_k \partial \alpha_l}
              (\valpha)
	    -
            \frac{\partial^2 \ell}
                 {\partial \alpha_k \partial \alpha_l}
              (\valpha')
          \right|
            \leq (M_1 \, X_{kl} + \bar M_1 \, \tilde p_{kl})
	         \, \| \valpha' - \valpha \|_\infty,
        \)
	where
	\(
          M_1 = F_2 \, \varepsilon_0
	\)
	and
	\(
	  \bar M_1 = 2r \, (1 + \bar F_0 + \bar F_1),
	\)
	provided $k \neq l$;
     \item
        \(
          \left|
            \frac{\partial^2 \ell}{\partial \alpha_k^2}(\valpha)
	    -
            \frac{\partial^2 \ell}{\partial \alpha_k^2}(\valpha')
          \right|
            \leq  M_2 \, X_{k+}
	         \, \| \valpha' - \valpha \|_\infty,
        \)
	where
	\(
	  M_2 = M_1 + \bar M_1.
	\)
    \end{enumerate}
  \end{lemma}

Lemma~\ref{L:E-norm-bounds} bounds the size of $\mE$, the relative error in approximating $\nabla^2 \ell(\vtalpha)$ by $\mH$, and the remaining Lemmas~\ref{L:inv-hess}--\ref{L:hess-lipschitz} verify that the necessary hypotheses are satisfied in order to apply Kantorovich's Theorem to bound the error in Newton's method.

  \begin{lemma}\label{L:E-norm-bounds}
    \(
       \| \mE \|_\infty \leq B_0 \, \varepsilon_0,
    \)
    where
    \(
      B_0 = (3/2)(L_2 + \bar L_2 + L_3).
    \)
  \end{lemma}
  \begin{proof}
    Matrix $\mE$ is given by
    \(
      \mE
        =
          \mH^{-1}
          \big(\nabla^2 \ell(\vtalpha)
               - \mH
          \big).
    \)
    In light of Lemma~\ref{L:grad-hess-null-bounds} and the triangle inequality,
    \begin{align*}
      |\mE|
        &\leq
          \Big(\mD^{-1} + \frac{1}{2 X_{++}} \vone \vone^\trans\Big)
          \Big(
            L_2 \, \varepsilon_0 \, \mX
            + \frac{\bar L_2 \, \varepsilon_0}{X_{++}} \vd \vd^\trans
            + L_3 \, \varepsilon_0 \, \mD
          \Big) \\
        &=
          \varepsilon_0
          \Big(
            \frac{L_2 + 3 \bar L_2 + L_3}{2 X_{++}}
            \vone \vd^\trans
            +
            L_2 \,
            \mD^{-1} \mX
            +
            L_3 \, \mI
          \Big).
    \end{align*}
    Thus,
    \(
      \| \mE \|_\infty
        \leq
          (3/2) \,
          \varepsilon_0  \,
          (
              L_2
            + \bar L_2
            + L_3
          )
        =
          B_0 \,
          \varepsilon_0.
    \)
   \end{proof}

 \begin{lemma}\label{L:inv-hess}
   If
   \(
     B_0 \, \varepsilon_0 < 1,
   \)
   then
   \(
     \big\|
       \big[
           \mD^{-1}
           \nabla^2 \ell (\vtalpha)
       \big]^{-1}
    \big\|_\infty
       \leq
       \varkappa,
  \)
  where
  \(
    \varkappa =
       (3/2) / (1 - B_0 \, \varepsilon_0).
  \)
 \end{lemma}
 \begin{proof}
   This follows from Lemma~\ref{L:E-norm-bounds} and Lemma~2.3.3
   in \cite{golub1996matrix}, with the bound
   \[
     \| [\mD^{-1} \nabla^2 \ell(\vtalpha)]^{-1} \|_\infty
       \leq \| (\mI + \mE)^{-1} \|_\infty \, \| \mH^{-1} \mD \|_\infty.
   \]
 \end{proof}

  \begin{lemma}\label{L:newton-step}
  If
  \(
    B_0 \, \varepsilon_0 < 1,
  \)
  then
   \(
     \big\|
       [\nabla^2 \ell(\vtalpha)]^{-1}
       [\nabla   \ell(\vtalpha)]
     \big\|_\infty
       \leq
       \delta,
   \)
   where
   \(
     \delta
       =
       L_1 \, \varkappa \, \varepsilon_0.
   \)
  \end{lemma}

  \begin{proof}
   In light of Lemma~\ref{L:grad-hess-null-bounds}, we know that
   \(
     |\nabla \ell(\vtalpha)| \leq L_1 \, \varepsilon_0 \, \vd.
   \)
   The result now follows after bounding
   \[
     \| [\nabla^2 \ell(\vtalpha)]^{-1} [ \nabla \ell(\vtalpha)] \|_\infty
       \leq
         \| [ \mD^{-1} \nabla^2 \ell(\vtalpha)]^{-1} \|_\infty
         \,
         \| \mD^{-1} \nabla \ell(\vtalpha)] \|_\infty.
   \]
 \end{proof}

 \begin{lemma}\label{L:hess-lipschitz}
   If $\valpha, \valpha' \in \mathcal{N}_r$, then
   \[
     \big\| \mD^{-1} [\nabla^{2} \ell (\valpha)]
     -  \mD^{-1} [\nabla^{2} \ell (\valpha')]
     \big\|_\infty
       \leq
         \lambda
         \,
         \| \valpha - \valpha' \|_\infty,
   \]
   where
   \(
     \lambda = 2 M_2.
   \)
 \end{lemma}
 \begin{proof}
   This is a direct consequence of Lemma~\ref{L:hessian-diff}.
 \end{proof}

\subsection*{Proof of Corollary~\ref{C:null-pij-approx}}

We aim to show that $|\hat p_{ij} - \tilde p_{ij}| / \tilde p_{ij} \leq 24 \, (C_0 + 1) \, \varepsilon_0$ under the conditions of Theorem~\ref{T:null-mle-approx}.
Write
\[
  \Big|\frac{\hat p_{ij} - \tilde p_{ij}}{\tilde p_{ij}}\Big|
    = \big|\exp\{(\hat \alpha_i - \tilde \alpha_i)
         + (\hat \alpha_j - \tilde \alpha_j)
         + \varepsilon_{ij}(\hat \alpha_i,\hat \alpha_j)\}
         - 1\big|.
\]
Using the bound $\tilde \alpha_i \leq (1/2) \log \varepsilon_0$ from the hypothesis of Theorem~\ref{T:null-mle-approx}, along with the result of the theorem that
  \(
    \| \vhalpha - \vtalpha \|_\infty
      \leq  C \, \varepsilon_0,
  \)
we may write
\[
  \hat \alpha_i \leq \frac{1}{2} \log \varepsilon_0 + C \, \varepsilon_0.
\]
By Assumption~\ref{A:small-eps} and the above, it thus follows that
\[
|\varepsilon_{ij}(\hat \alpha_i,\hat \alpha_j)|
  \leq C_0 \exp(\hat \alpha_i + \hat \alpha_j)
  \leq C_0 \, \varepsilon_0 \, \exp(2 C \varepsilon_0).
\]
Now, since $\varepsilon_0 \leq \bar \varepsilon_0 = \{15 \, (C_0 + 1)\}^{-2}$ and with $C = 10 \, (C_0 + 1)$, we obtain after
simplification that
\[
  |(\hat \alpha_i - \tilde \alpha_i)
         + (\hat \alpha_j - \tilde \alpha_j)
         + \varepsilon_{ij}(\hat \alpha_i,\hat \alpha_j)|
    \leq \text{21$\cdot$1} \, \varepsilon_0 \, (C_0 + 1)
    \leq \log \text{1$\cdot$1}.
\]
The result then follows by using the bound $|e^x - 1| \leq |x \, e^x|$.

\subsection*{Proof of Corollary~\ref{C:null-ll-approx}}

Corollary~\ref{C:null-ll-approx} claims $|\hat \ell - \tilde \ell| / |\tilde \ell | \leq 49 \, (C_0 + 1) \, \varepsilon_0$.  From Corollary~\ref{C:null-pij-approx} we obtain the bounds
\begin{gather*}
  |\log \hat p_{ij} - \log \tilde p_{ij}| \leq C_1 \, \varepsilon_0, \\
  |\log (1 - \hat p_{ij}) - \log (1 - \tilde p_{ij})|
    \leq C_1 \, \frac{\varepsilon_0}{1 - \varepsilon_0} \, \tilde p_{ij},
\end{gather*}
with $C_1 = 24 \, (C_0 + 1)$.
Therefore,
\[
  |\hat \ell - \tilde \ell|
    \leq
      C_1 \, \varepsilon_0 \sum_{i<j} X_{ij}
      +
      C_1 \, \frac{\varepsilon_0}{1 - \varepsilon_0} \sum_{i<j} (1 - X_{ij}) \, \tilde p_{ij}
    \leq
      C_1 \, \frac{\varepsilon_0}{1 - \varepsilon_0} X_{++}.
\]
Now write
\begin{align*}
  |\tilde \ell|
    &= -\sum_{i=1}^n \tilde \alpha_i \, X_{i+} - \sum_{i < j} (1 - X_{ij}) \log(1 - \tilde p_{ij}) \\
    &= \frac{X_{++}}{2} \log X_{++} - \sum_{i = 1}^{n} X_{i+} \log X_{i+}
        - \sum_{i < j} (1 - X_{ij}) \log(1 - \tilde p_{ij}) \\
    &\geq \frac{X_{++}}{2} \log X_{++} - X_{++} \log \frac{X_{++}}{n} \\
    &= \frac{X_{++}}{2} \log \frac{n^2}{X_{++}} \\
    &\geq \frac{X_{++}}{2}.
\end{align*}
Finally, putting the two bounds together and using that $\varepsilon_0 \leq \bar \varepsilon_0 \leq 15^{-2}$, we get
\[
  \frac{|\hat \ell - \tilde \ell|}{|\tilde \ell|}
     \leq 2 \, C_1 \, \frac{\varepsilon_{0}}{1 - \varepsilon_{0}}
     \leq 49 \, (C_0 + 1) \, \varepsilon_0.
\]

\bibliographystyle{biometrika}
\bibliography{refs}

\end{document}